\documentclass{amsart}
\usepackage{lmodern}
\usepackage{microtype}
\usepackage{amsmath}
\usepackage{amssymb}
\usepackage{amsthm}
\usepackage{amscd}
\usepackage{hyperref}
\usepackage{cite}
\usepackage{bm}

\makeatletter
\newtheorem*{rep@theorem}{\rep@title}
\newcommand{\newreptheorem}[2]{%
\newenvironment{rep#1}[1]{%
 \def\rep@title{#2 \ref{##1}}%
 \begin{rep@theorem}}%
 {\end{rep@theorem}}}
\makeatother

\newtheorem{theorem}{Theorem}[section]
\newreptheorem{theorem}{Theorem}
\newtheorem{lemma}[theorem]{Lemma}

\theoremstyle{definition}
\newtheorem{definition}[theorem]{Definition}

\makeatletter

\def\dotminussym#1#2{%
  \setbox0=\hbox{$\m@th#1-$}%
  \kern.5\wd0%
  \hbox to 0pt{\hss\hbox{$\m@th#1-$}\hss}%
  \raise.6\ht0\hbox to 0pt{\hss$\m@th#1.$\hss}%
  \kern.5\wd0}

\mathchardef\mhyphen="2D

\begin{document}

\title{Limits of Sequences of Markov Chains}
\author{Henry Towsner}
\date{\today}
\thanks{Partially supported by NSF grant DMS-1340666.}

\begin{abstract}
We study the limiting object of a sequence of Markov chains analogous to the limits of graphs, hypergraphs, and other objects which have been studied.  Following a suggestion of Aldous, we assign to a convergent sequence of finite Markov chains with bounded mixing times a unique limit object: an infinite Markov chain with a measurable state space.  The limits of the Markov chains we consider have discrete spectra, which makes the limit theory simpler than the general graph case, and illustrates how the discrete spectrum setting (sometimes called ``random-free'' or ``product measurable'') is simpler than the general case.
\end{abstract}

\maketitle

\section{Introduction}

Suppose we have a continuous time Markov chain with a very large, but finite, number of states.  (We are interested in the case where the chain is reversible and time-homogeneous.)  We would expect that the chain resembles a chain with an infinite measurable state space.  In this paper, we make this precise.  To any finite continuous time Markov chain we can associate a partially exchangeable array of random variables by randomly sampling a sequence of points from the space and taking the transition rates.  To a sequence of finite chains we associate an infinite Markov chain whose associated partially exchangeable array is a limit (in distribution) of the arrays of the finite chains.  We further show that (after some refinement) these infinite chains are essentially unique.

In order to ask for a sequence of Markov chains to have a limit, we need the sequence to be ``bounded'' in a some sense.  Following a suggestion by Aldous \cite{aldousnote}, we assume the mixing of the sequence is uniformly bounded (that is, for each time $t$ there is a bound $B_t$ such that the mixing of each chain at time $t$ is bounded by $B_t$).  Given such a sequence, we identify the sequence with an infinite Markov chain and show that the statistical behavior of the finite chains converges to the statistical behavior of this infinite chain.  This confirms a conjecture of Aldous on the existence of such a compactification.  Under suitable assumptions on the infinite chain, we identify it up to isomorphism.

After some preliminaries, Theorem \ref{thm:main} states the existence half of our main result.  We prove this in Section 5.  In Section 6 we show the converse, that any infinite Markov chain of the kind we consider is the limit of some finite sequence.  In Section 7 we give the corresponding uniqueness result: we show how to refine our infinite Markov chains to a (potentially) infinite Markov chain with additionl properties which is uniquely determined by the corresponding partially exchangeable array.

Similar results for graphs and hypergraphs have been known for several decades \cite{aldous:MR637937,hoover:arrays} and have recently been extensively studied \cite{borgs:MR2455626,lovasz:MR2306658,lovasz:MR2274085,diaconis:MR2463439,aldous:MR637937,hoover:arrays,kallenberg:MR2161313,austin08,hrushovski,goldbring:_approx_logic_measure,MR2964622,tao07} under various names, especially as \emph{graph limits} or, as here, \emph{ultraproducts}.  For example, a similar analysis was recently given by Elek \cite{2012arXiv1205.6936E} in the setting of metric measure spaces.  Gromov \cite{MR2307192} identified convergent sequences of metric measure spaces with certain partially exchangeable arrays of random variables.  Elek identifies each such array, essentially uniquely, with an infinitary object (a ``quantum metric measure space'').

Markov chains turn out to be simpler than these other cases in one important respect: our boundedness assumption implies that the limit has a discrete spectrum.  For graphs, sequences with discrete spectra have been a particular topic of interest (these are the ``random-free'' graph limits of \cite{MR3073488,MR2815610}, and the objects the author has called ``$\mathcal{B}_{2,1}$-measurable'' in \cite{goldbring:_approx_logic_measure,henry12:_analy_approac_spars_hyper,2013arXiv1312.4882T}), especially because of their connection to the Szemer\'edi regularity lemma.  In our case it allows us to avoid certain complications compared to the graph case.

\section{Three Descriptions}

\subsection{Finite State Markov Chains}

We first recall the basic definitions for the finite objects we will be considering.

\begin{definition}
  A \emph{finite state continuous time homogeneous Markov chain} consists of a finite \emph{state space} $\Omega$, a family of $\Omega$-valued random variables $\{\mathbf{X}(t)\}_{t\geq 0}$, and a \emph{transition rate matrix} $\mathbf{Q}$ such that:
  \begin{itemize}
  \item Each non-diagonal entry $\mathbf{Q}(\omega,\omega')$ with $\omega\neq\omega'$ is non-negative,
  \item The rows of $\mathbf{Q}$ sum to $0$,
  \item for any $\omega,\omega'\in \Omega$ and any $s,t>0$,
\[\mathbb{P}(\mathbf{X}(s+t)=\omega\mid\mathbf{X}(s)=\omega')=e^{t\mathbf{Q}}(\omega,\omega').\]
  \end{itemize}
\end{definition}
Given the matrix $\mathbf{Q}$, we associate the family of \emph{transition probability} matrices $\mathbf{P}_t=e^{t\mathbf{Q}}$.

For the remainder of this paper we will use ``finite Markov chain'' to mean a finite state continuous time homogeneous Markov chain.

\begin{definition}
A probability distribution $\pi$ on $\Omega$ is a \emph{stationary distribution} if for every $\omega'\in\Omega$, $\sum_\omega\pi(\omega)\mathbf{Q}(\omega,\omega')=\pi(\omega')$.
\end{definition}

\begin{definition}
  We say the Markov chain is \emph{reversible} if there is a stationary distribution $\pi$ on $\Omega$ such that for every $\omega,\omega'$,
\[\pi(\omega)\mathbf{Q}(\omega,\omega')=\pi(\omega')\mathbf{Q}(\omega',\omega).\]

  We say a Markov chain is \emph{irreducible} if every entry in $\mathbf{P}_t$ is strictly positive for some (equivalently, for every) $t>0$.
\end{definition}
It is standard that an irreducible Markov chain has at most one stationary distribution $\pi$ and $\pi(\omega)>0$ for all $\omega\in\Omega$.

In order to have well-behaved limits, we need some type of boundedness condition.  An easy example illustrates why this is necessary: consider a sequence of reversible Markov chains consisting of two points $\omega_0,\omega_1$ (each with measure $1/2$ in the stationary distribution) where the transition rate in the $n$-th Markov chain, $\mathbf{Q}_{(n)}(\omega_0,\omega_1)=1/n$.  That is, as we consider later chains in the sequence, the chain mixes more slowly.  In the limit, the mixing approaches $0$, and indeed, in any limit object the two points would not mix at all, causing the limit to be reducible.\footnote{If we wish to insist on our Markov chains having a number of states approaching infinity, replace $\omega_0$ and $\omega_1$ with blocks of states $\Omega_{(n),0},\Omega_{(n),1}$ where the sets are growing as $n$ grows, the transition rates within $\Omega_0$ and $\Omega_1$ are constant, but the transition rates between any $\omega_0\in\Omega_{(n),0}$ and $\omega_1\in\Omega_{(n),1}$ is shrinking quickly enough in $n$.}

Aldous proposes \cite{aldousnote} that this be addressed by normalizing the mixing time:
\begin{definition}
  We define $G(t)$, the \emph{mixing at time $t$} (relative to $\pi$) to be
\[\sum_{\omega}\mathbf{P}_t(\omega,\omega).\]
We say the chain is \emph{normalized} if $G(1)=2$.
\end{definition}
Note that an irreducible Markov chain is fully mixed precisely when $G(t)=1$.  (Even if this doesn't happen in any finite time, it could happen in the limit.)

Aldous points out that this is still not enough to ensure reasonable limit objects because a sequence of normalized Markov chains might experience the $L^2$ cutoff phenomenon (see \cite{MR2375599,MR841111}).  If the cutoff phenomenon occurs, we have $\lim_n G_{(n)}(t)=\infty$ whenever $t<1$ and $\lim_n G_{(n)}(t)=1$ for $t>1$: as $n$ approaches infinity, the mixing happens in a shorter and shorter window around $t=1$.  Following Aldous' suggestion, we work with sequences of chains where $L^2$ cutoff does not occur.  Equivalently:
\begin{definition}
  A \emph{bounded sequence of Markov chains} is a sequence of finite Markov chains $\Omega_{(n)},\mathbf{Q}_{(n)}$ such that:
  \begin{itemize}
  \item Each chain is irreducible, reversible, and normalized,
  \item For each $t>0$ there is a $B_t$ such that $G_{(n)}(t)\leq B_t$ for all $n$.
  \end{itemize}
\end{definition}

Note that, other than the boundedness of $G_{(n)}(t)$, there are no convergence requirements on a bounded sequence.  Thus we will pass to subsequences of a given bounded sequence in order to have suitable limit objects.

A sequence of Markov chains can have different portions of its mixing happen at different time scales: consider the sequence of reversible Markov chains consisting of four points, $\omega_{00},\omega_{01},\omega_{10},\omega_{11}$ (each with measure $1/4$ in the stationary distribution) where the transition rate in the $n$-th Markov chain $\mathbf{Q}_{(n)}(\omega_{i0},\omega_{i1})=n$ while $\mathbf{Q}_{(n)}(\omega_{0i},\omega_{1i})=1$ (and, for simplicity, $\mathbf{Q}_{(n)}(\omega_{00},\omega_{11})=0$).  If we take $B_0=\{\omega_{00},\omega_{01}\}$ and $B_1=\{\omega_{10},\omega_{11}\}$, the mixing between $B_0$ and $B_1$ has a fixed rate while the mixing within the sets $B_0$ and $B_1$ happens faster and faster.  (Indeed, when we take the limit object, the sets $B_0$ and $B_1$ will each become an indistinguishable blob: in the limit, we can't distinguish $\omega_{00}$ from $\omega_{01}$ because they mix instantly.)

These Markov chains have the property that there is a $t_0$ independent of $n$ so that $G_{(n)}(t_0)\approx 2$.  (In other words, they already approximately satisfy a normalization condition of the form $G(a)=b$ for some $a>0$ and some $b\in(1,\infty)$.)  In the limit, ``some of the mixing''---the mixing internal to $B_0$ and $B_1$---happens very quickly, but the ``largest scale'' of mixing, the mixing between $B_0$ and $B_1$, happens in finite time: when $t$ is very small (and $n$ large), $\mathbf{P}_{(n),t}(\omega_{00},\omega_{11})$ is small while $\mathfrak{P}_{(n),1/t}(\omega_{00},\omega_{11})$ is close to $1/4$.

This is what we are ensuring by normalizing and bounding $G$: that this largest scale of mixing happens at the same scale as $t$.  If we dropped the boundedness requirement, we would be including limits which go from a discrete collection of completely unmixed blocks at every time $t<1$ to being completely mixed at every time $t>1$.  If we dropped the normalization requirement, we could have chains which never fully mix (if $\lim_{n\rightarrow\infty}G_{(n)}(t)\rightarrow \infty$ for every $t$) or which have already mixed at every $t>0$ (if $\lim_{n\rightarrow\infty}G_{(n)}(t)\rightarrow 1$ for every $t$).

(We do still retain one anomolous case: where $\lim_{n\rightarrow\infty}G_{(n)}(t)=2$ for \emph{all} times $t$.  In the limit, this chain must be constant (because $G$ is not changing) and therefore reducible.  But the presence of this case will not interfere with any of our arguments; we could get rid of it by adding the assumption that $\lim_{t\rightarrow\infty}\lim_{n\rightarrow\infty}G_{(n)}(t)=1$.)

\subsection{Pseudofinite Chains}

Our second notion is a specific kind of infinite space Markov chains.  These are in some respects simpler to work with than arbitrary infinite space continuous time Markov chains, so to distinguish them, we call them \emph{pseudofinite continuous time Markov chains}, or just \emph{pseudofinite Markov chains}.  (The term pseudofinite here comes from model theory, where it refers to a model which has the same first-order logical properties as a finite model.  This will be true of our pseudofinite Markov chains which, as we will see, are essentially equivalent to convergent limits of finite Markov chains.)

\begin{definition}
By a \emph{pseudofinite continuous time Markov chain}, we mean a probability space $(\Omega,\mathcal{B},\pi)$ and, for each $t\in\mathbb{R}^{>0}$, a measurable function
\[\hat p_t:\Omega^2\rightarrow\mathbb{R}^{\geq 0}\]
such that (taking all integrals over $\pi$):
\begin{itemize}
\item (Stochasticity) For every $t>0$ and almost every $\omega$, $\int \hat p_t(\omega,\omega')d\omega'=1$,
\item (Symmetry) For every $t>0$ and almost every $\omega,\omega'$, $\hat p_t(\omega,\omega')=\hat p_t(\omega',\omega)$,
\item (Chapman-Kolmogorov) For every $s,t>0$ and almost every $\omega,\omega'$, $\hat p_{s+t}(\omega,\omega')=\int \hat p_s(\omega,\xi)\hat p_t(\xi,\omega')d\xi$,
\item (Diagonal Chapman-Kolmogorov) For every $s,t>0$ and almost every $\omega$, $\hat p_{s+t}(\omega,\omega)=\int \hat p_s(\omega,\xi)\hat p_t(\xi,\omega)d\xi$,
\item (Boundedness) For every $t>0$, $\int\hat p_t(\omega,\omega)d\omega$ is finite,
\item (Normality) $\int\hat p_1(\omega,\omega)d\omega=2$,
\item (Continuity) The function $t\mapsto \hat p_t$ is continuous with respect to the $L^2$ norm---that is, for every $t>0$, $\lim_{s\rightarrow t}||\hat p_t-\hat p_s||_{L^2(\pi\times\pi)}=0$.
\end{itemize}

Given a pseudofinite continuous time Markov chain $\hat p_t$, we write $G(t)=\int \hat p_t(\omega,\omega)d\omega$.
\end{definition}
Below, when discussing pseudofinite Markov chains, we will generally assume that integrals are over $\pi$ and that the $L^2$ space of interest is $L^2(\pi\times\pi)$.

Even stating the continuity property suggests that each $\hat p_t$ has bounded $L^2$ norm, and this follows from diagonal Chapman-Kolmogorov, symmetry, and boundedness:
\begin{lemma}
  $||\hat p_t||_{L^2}=\sqrt{G(2t)}$
\end{lemma}
\begin{proof}
  \begin{align*}
    ||\hat p_t||^2_{L^2}
&=\iint \hat p_t^2(\omega,\omega') d\omega' d\omega \\
&=\iint \hat p_t(\omega,\omega')\hat p_t(\omega',\omega) d\omega' d\omega\\
&=\int \hat p_{2t}(\omega,\omega) d\omega\\
&=G(2t).
  \end{align*}
\end{proof}

Note that, for each $t$, $\hat p_t$ induces an operator $\widehat{P}_t$ on the $L^2$ functions by
\[\widehat{P}_t(f)(\omega)=\int f(\xi)\hat p_t(\omega,\xi)d\xi.\]
Stochasticity and symmetry mean that the operator preserves the $L^1$ norm, and the Chapman-Kolmogorov condition ensures that the action is actually a flow---$\widehat{P}_t\circ\widehat{P}_s=\widehat{P}_{s+t}$.

\subsection{Properties of Pseudofinite Chains}

In this subsection we show that the usual eigenvector decomposition can be recovered for pseudofinite continuous space Markov chains, by essentially the usual proof.

$\widehat{{P}_t}$ is the Hilbert-Schmidt operator corresponding to $\hat p_t$.  Since $\hat p_t$ is symmetric, $\widehat{{P}_t}$ is symmetric as well.  The spectral theorem tells us that for each $\widehat{{P}_t}$, there is a basis for the $L^2$ functions consisting of eigenvectors of $\widehat{{P}_t}$.  Clearly the function which is constantly $1$ is an eigenvector, with eigenvalue $1$ (by stochasticity of $\hat p$).

\begin{lemma}
  $\widehat{P}_t$ is positive semidefinite.
\end{lemma}
\begin{proof}
  Let $\nu$ be any $L^2$ function.  Then
\begin{align*}
  \int\nu(\omega)\widehat{{P}_t}(\nu)(\omega)d\omega
&=\iint\nu(\omega)\nu(\xi)\hat{p}_t(\omega,\xi)d\xi\, d\omega\\
&=\iiint \nu(\omega)\nu(\xi)\hat p_{t/2}(\omega,\zeta)\hat p_{t/2}(\zeta,\xi)d\zeta\,d\xi\,d\omega\\
&=\int\left(\int\hat p_{t/2}(\omega,\zeta)\nu(\omega)d\omega\right)^2d\zeta\\
&\geq 0.
\end{align*}
\end{proof}

\begin{lemma}
  Any eigenvalue of $\widehat{P}_t$ is in the interval $[0,1]$.
\end{lemma}
\begin{proof}
If $\nu$ is an eigenfunction with eigenvalue $\gamma$, the previous lemma tells us
\[0\leq \int\nu(\omega)\widehat{P}_t(\nu)(\omega)d\omega=\gamma||\nu||^2_{L^2},\]
so $0\leq\gamma$.

On the other hand,
\begin{align*}
  \gamma||\nu||^2_{L^2}
&=\int \nu(\omega)\widehat{P}_t(\nu)\omega d\mu\\
&=\iint \nu(\omega)\hat p_t(\xi,\omega)\nu(\xi)d\xi\,d\omega\\
&=\iint (\nu(\omega)\sqrt{\hat p_t(\xi,\omega)})(\nu(\xi)\sqrt{\hat p_t(\xi,\omega)})d\xi\,d\omega\\
&\leq\sqrt{\iint \nu^2(\omega)\hat p_t(\xi,\omega)d\xi\,d\omega\iint \nu^2(\xi)\hat p_t(\xi,\omega)d\xi\,d\omega}\\
&=||\nu||^2_{L^2},
\end{align*}
so $\gamma\leq 1$.
\end{proof}

\begin{lemma}\label{thm:close_rat}
  For each $t$, $\hat p_t(\omega,\omega')=\sum_i\lambda_i\nu_i(\omega)\nu_i(\omega')$ where the $\nu_i$ are eigenvectors forming an orthonormal basis for the support of $\widehat{P}_t$ and the $\lambda_i$ are the corresponding eigenvalues.
\end{lemma}
\begin{proof}
Let $q=\hat p_t(\omega,\omega')-\sum_{i}\lambda_i\nu_i(\omega)\nu_i(\omega')$.  If $||q||_{L^2}>0$ then the operator $Q(f)(\omega)=\int f(\xi)q(\omega,\xi)d\omega$ has an eigenvector $\nu$, and $\nu$ must be orthogonal to all the $\nu_i$.  But this means $Q(\nu)=\widehat{P}_t(\nu)$, so $\nu$ is an eigenvector of $\hat p_t$, so $Q(\nu)$ must be $0$.
\end{proof}

\begin{lemma}
If $\hat p_t(\omega,\omega')=\sum_i\lambda_i^t\nu_i(\omega)\nu_i(\omega')$ then $\hat p_{nt}(\omega,\omega')=\sum_i\lambda_i^{nt}\nu_i(\omega)\nu_i(\omega')$.
\end{lemma}
\begin{proof}
  By induction on $n$.  We have
  \begin{align*}
    \hat p_{(n+1)t}(\omega,\omega')
&=\int \hat p_{nt}(\omega,\xi)\hat p_t(\xi,\omega')d\xi\\
&=\int \sum_i\lambda_i^{nt}\nu_i(\omega)\nu_i(\xi)\sum_j\lambda_j^{t}\nu_j(\xi)\nu_j(\omega')d\xi\\
&=\sum_{i,j}\lambda_i^{nt}\lambda_j^t\nu_i(\omega)\nu_j(\omega')\int \nu_i(\xi)\nu_j(\xi)d\xi\\
&=\sum_{i}\lambda_i^{(n+1)t}\nu_i(\omega)\nu_i(\omega')\\
  \end{align*}
using the fact that $\int \nu_i(\xi)\nu_j(\xi)d\xi=1$ if $i=j$ and $0$ otherwise.
\end{proof}

This ensures that the eigenvectors decompositions for $t$ and $qt$ agree when $q$ is rational.  Continuity then gives us the same statement for all $t$.

\begin{lemma}\label{thm:close_all}
  For every $t>0$, $\lim_{s\rightarrow t}||\sum_i\lambda_i^t\nu(\omega)\nu(\omega')-\sum_i\lambda_i^s\nu(\omega)\nu(\omega')||_{L^2}=0$.
\end{lemma}
\begin{proof}
  Let $t>0$ and $\epsilon>0$ sufficiently small be given.  Choose $k$ large enough that $\sum_{i>k}\lambda_i^{t/2}<\epsilon/3$.  When $s>t/2$, we have
\[(\sum_{i>k}\lambda_i^s)^{1/s}\leq (\sum_{i>k}\lambda_i^{t/2})^{2/t}<(\epsilon/3)^{2/t}\]
and so $\sum_{i>k}\lambda_i^s<(\epsilon/3)^{2s/t}<\epsilon/3$.
  Additionally, when $s$ is close to $t$, for each $i\leq k$ we have $1-\lambda_i^{s-t}<\epsilon/3k$.  Then
\begin{align*}
||\sum_i\lambda_i^t\nu(\omega)\nu(\omega')-\sum_i\lambda_i^s\nu(\omega)\nu(\omega')||_{L^2}
&\leq ||\sum_{i\leq k}\lambda_i^t\nu_i(\omega)\nu_i(\omega')-\sum_{i\leq k}\lambda_i^s\nu_i(\omega)\nu_i(\omega')||_{L^2}\\
&\ \ \ \ +||\sum_{i>k}\lambda_i^t\nu_i(\omega)\nu_i(\omega')||_{L^2}\\
&\ \ \ \ +||\sum_{i>k}\lambda_i^s\nu_i(\omega)\nu_i(\omega')||_{L^2}\\
&\leq ||\sum_{i\leq k}\lambda_i^t(1-\lambda_i^{s-t})\nu_i(\omega)\nu_i(\omega')||_{L^2}+\epsilon/3+\epsilon/3\\
&\leq\sum_{0<i\leq k}\frac{\epsilon}{2k}\lambda_i^t||\nu_i(\omega)\nu_i(\omega')||_{L^2}+\epsilon/3\\
&\leq\epsilon.
\end{align*}
\end{proof}

In particular, taking $\hat p_1(\omega,\omega')=\sum_i\lambda_i\nu_i(\omega)\nu_i(\omega')$, we have
\[||\hat p_t(\omega,\omega')-\sum_i\lambda_i^t\nu(\omega)\nu(\omega')||_{L^2}=0\]
for all $t$: for rational $t$ this follows from Lemma \ref{thm:close_rat}, and then for arbitrary $t$ this follows because for every $\epsilon>0$ we have
\[  ||\hat p_t(\omega,\omega')-\sum_i\lambda_i^t\nu(\omega)\nu(\omega')||_{L^2}
\leq||\hat p_t(\omega,\omega')-\hat p_s(\omega,\omega')||_{L^2}+||\sum_i\lambda_i^t\nu(\omega)\nu(\omega')-\sum_i\lambda_i^s\nu(\omega)\nu(\omega')||_{L^2},\]
which can be made arbitrarily small by choosing $s$ to be a rational number near $t$.

\subsection{Exchangeable Arrays}

Our third notion discards the explicit description of a Markov chain to focus on the statistical properties of the densities. 

\begin{definition}
  Let $(\Omega,\mathcal{B},\pi),\hat p_t$ be a pseudofinite Markov chain.  The \emph{density array corresponding to $(\Omega,\mathcal{B},\pi),\hat p_t$} is the collection of random variables $(\mathbf{X}_{i,j}(t))_{i,j\in\mathbb{N}}$ obtained by selecting an i.i.d. random sequence $(\omega_i)_{i\in\mathbb{N}}$ from $\Omega$ according to $\pi$ and setting $\mathbf{X}_{i,j}(t)=\hat p_t(\omega_i,\omega_j)$.
\end{definition}

Since the $\omega_i$ are i.i.d., it is easy to see that these random variables are partially exchangeable and dissociated:
\begin{definition}
  An array of random variables $(\mathbf{X}_{i,j})_{i,j\in\mathbb{N}}$ is \emph{partially exchangeable} if whenever $\sigma:[0,n]\rightarrow[0,n]$ is a permutation, the joint distribution of $(\mathbf{X}_{i,j})_{i,j\in[0,n]}$ is identical to the joint distribution of $(\mathbf{X}_{\sigma(i),\sigma(j)})_{i,j\in[0,n]}$.

  An array is \emph{dissociated} if whenever $S$ and $T$ are disjoint, $(\mathbf{X}_{i,j})_{i,j\in S}$ is independent of $(\mathbf{X}_{i,j})_{i,j\in T}$.
\end{definition}

We wish to identify those arrays of random variables which can be obtained in this way.  Unsurprisingly, most properties amount to translations of the corresponding properties of a pseudofinite Markov chain.  Some of these properties (particularly the Chapman-Kolmogorov property) are awkward to express directly as a property of an array of random variables; the definition is justified by Theorem \ref{thm:pseudo_to_array} below, which shows how each property relates to the corresponding property of a pseudofinite Markov chain.

\begin{definition}
A \emph{Markov chain density array} is an array $(\mathbf{X}_{i,j})_{i,j\in\mathbb{N}}$ of $\left(\mathbb{R}^{>0}\right)^{\mathbb{R}^{>0}}$-valued random variables such that:
  \begin{itemize}
  \item (Exchangeability) The array is partially exchangeable,
  \item (Dissociated) The array is dissociated,
  \item (Independence) For every $n$, $(\mathbf{X}_{i,j})_{i,j<n}$ is independent of $(\mathbf{X}_{i,j})_{i,j\geq n}$,
  \item (Stochasticity) For every $t>0$, $\mathbb{E}(2\mathbf{X}_{0,1}(t)-\mathbf{X}_{0,1}(t)\mathbf{X}_{0,2}(t))=1$,
  \item (Symmetry) For every $t$, with probability $1$, $\mathbf{X}_{0,1}(t)=\mathbf{X}_{1,0}(t)$,
  \item (Chapman-Kolmogorov) For every $s,t>0$,
\[\mathbb{E}(\mathbf{X}_{0,1}(s+t)^2-2\mathbf{X}_{0,1}(s+t)\mathbf{X}_{0,2}(s)\mathbf{X}_{2,1}(t)+\mathbf{X}_{0,2}(s)\mathbf{X}_{0,3}(s)\mathbf{X}_{2,1}(t)\mathbf{X}_{3,1}(t))=0,\]
\item (Diagonal Chapman-Kolmogorov) For every $s,t>0$,
\[\mathbb{E}(\mathbf{X}_{0,0}(s+t)^2-2\mathbf{X}_{0,0}(s+t)\mathbf{X}_{0,2}(s)\mathbf{X}_{2,0}(t)+\mathbf{X}_{0,2}(s)\mathbf{X}_{0,3}(s)\mathbf{X}_{2,0}(t)\mathbf{X}_{3,0}(t))=0,\]
\item (Boundedness) $\mathbb{E}(\mathbf{X}_{0,0}(t))$ is finite for all $t$,
\item (Normality) $\mathbb{E}(\mathbf{X}_{0,0}(1))=2$,
  \item (Continuity) For every $t>0$, $\lim_{s\rightarrow t}\mathbb{E}(\left(\mathbf{X}_{0,1}(t)-\mathbf{X}_{0,1}(s)\right)^2)=0$.
  \end{itemize}

  If for each $n$, $(\mathbf{X}_{(n),i,j})_{i,j\in\mathbb{N}}$ is a Markov chain density array, we say $(\mathbf{X}_{(n),i,j})$ \emph{converges in distribution} to $(\mathbf{X}_{*,i,j})_{i,j\in\mathbb{N}}$ if for every $k$, the finite matrix of random variables $(\mathbf{X}_{(n),i,j})_{i,j\leq k}$ converges in distribution to $(\mathbf{X}_{*,i,j})_{i,j\leq k}$.


\end{definition}

\begin{theorem}\label{thm:pseudo_to_array}
  Let $(\Omega,\mathcal{B},\pi), \hat p_t$ be a pseudofinite Markov chain.  The density array corresponding to $(\Omega,\mathcal{B},\pi), \hat p_t$ is a Markov chain density array.
\end{theorem}
\begin{proof}
  As noted above, exchangeability and dissociation follows immediately from the fact that the $\omega_i$ are chosen i.i.d..  Independence follows as well.

In general, suppose we take any function depending on finitely many values of the form $\mathbf{X}_{i,j}(t)$---that is,
\[f(\mathbf{X}_{i_0,j_0}(t_0),\ldots,\mathbf{X}_{i_m,j_m}(t_m))\]
with $i_k,j_k\leq n$ for each $k\leq m$.  Then the expected value
\[\mathbb{E}(f(\mathbf{X}_{i_0,j_0}(t_0),\ldots,\mathbf{X}_{i_m,j_m}(t_m)))\]
is the average result of selecting $\omega_0,\ldots,\omega_n$ and calculating
\[f(\hat p_{t_0}(\omega_{i_0},\omega_{j_0}),\ldots,\hat p_{t_m}(\omega_{i_m},\omega_{j_m})).\]
That is,
\[\mathbb{E}(f(\mathbf{X}_{i_0,j_0}(t_0),\ldots,\mathbf{X}_{i_m,j_m}(t_m)))=\int\cdots\int f(\hat p_{t_0}(\omega_{i_0},\omega_{j_0}),\ldots,\hat p_{t_m}(\omega_{i_m},\xi_{j_m}))d\omega_0\cdots d\omega_n.\]

All other properties of a Markov chain density array follow by applying this for suitable choices of $f$.

Stochasticity holds since 
\begin{align*}
\mathbb{E}(2\mathbf{X}_{0,1}(t)-\mathbf{X}_{0,1}(t)\mathbf{X}_{0,2}(t))-1
&=\iiint 2\hat p_t(\omega_0,\omega_1)-\hat p_t(\omega_0,\omega_1)\hat p_t(\omega_0,\omega_2) - 1 d\omega_0 d\omega_1 d\omega_2\\
&=-\int \left[\int \hat p_t(\omega_0,\omega_1)d\omega_1-1\right]^2 d\omega_0\\
&=0.
\end{align*}

Symmetry holds since
\[\mathbb{E}((\mathbf{X}_{1,0}(t)-\mathbf{X}_{0,1}(t))^2)=\iint (\hat p_t(\omega,\omega')-\hat p_t(\omega',\omega))^2d\omega d\omega'=0\]
by the symmetry of $\hat p_t$.

Chapman-Kolmogorov holds since
\begin{align*}
  &\mathbb{E}(\mathbf{X}_{0,1}(s+t)^2-2\mathbf{X}_{0,1}(s+t)\mathbf{X}_{0,2}(s)\mathbf{X}_{2,1}(t)+\mathbf{X}_{0,2}(s)\mathbf{X}_{0,3}(s)\mathbf{X}_{2,1}(t)\mathbf{X}_{3,1}(t))\\
=&\iiiint \hat p_{s+t}(\omega_0,\omega_1)^2-2\hat p_{s+t}(\omega_0,\omega_1)\hat p_s(\omega_0,\omega_2)\hat p_t(\omega_2,\omega_1)\\
&\ \ \ \ \ \ \ \ +\hat p_s(\omega_0,\omega_2)\hat p_s(\omega_0,\omega_3)\hat p_t(\omega_2,\omega_1)\hat p_t(\omega_3,\omega_1)d\omega_0d\omega_1d\omega_2d\omega_3\\
=&\iint \hat p_{s+t}(\omega_0,\omega_1)^2-2\hat p_{s+t}(\omega_0,\omega_1)\int \hat p_s(\omega_0,\omega_2)\hat p_t(\omega_2,\omega_1)d\omega_2\\
&\ \ \ \ \ \ \ \ +\left[\int\hat p_s(\omega_0,\omega_2)\hat p_t(\omega_2,\omega_1)d\omega_2\right]^2d\omega_0d\omega_1\\
=&\iint \left[\hat p_{s+t}(\omega_0,\omega_1)-\int \hat p_s(\omega_0,\omega_2)\hat p_t(\omega_2,\omega_1)d\omega_2\right]^2d\omega_0d\omega_1\\
=&0
\end{align*}
by the Chapman-Kolmogorov property of $\hat p_t$.

Diagonal Chapman-Kolmogorov holds by the same argument, replacing $\omega_1$ with $\omega_0$.


Boundedness holds since
\[\mathbb{E}(\mathbf{X}_{0,0}(t))=\int \hat p_t(\omega,\omega)d\omega\]
which is finite by the boundedness of $\hat p_t$.  When $t=1$, this expectation is $2$ by the normality of $\hat p_t$.

For any $t>0$ and any $s$,
\[\mathbb{E}((\mathbf{X}_{0,1}(t)-\mathbf{X}_{0,1}(s))^2)=\iint (\hat p_t(\omega,\omega')-\hat p_s(\omega,\omega'))^2d\omega d\omega'=||\hat p_t-\hat p_s||_{L^2}^2.\]
Since the right side approaches $0$ as $s$ approaches $t$, the left side does as well.
\end{proof}

\section{Scaling Finite Markov Chains}

To compare finite Markov chains to density arrays, we want to first rescale according to the stationary distribution.

\begin{definition}
  Let $\Omega,\mathbf{Q}$ be a reversible, irreducible, finite Markov chain with stationary distribution $\pi$.  We define the \emph{scaled transition rate} and \emph{probability density} to be
\[\widehat{\mathbf{Q}}(\omega,\omega')=\frac{\mathbf{Q}(\omega,\omega')}{\pi(\omega')}\text{ and }\widehat{\mathbf{P}}_t(\omega,\omega')=\frac{\mathbf{P}_t(\omega,\omega')}{\pi(\omega')}.\]
\end{definition}

\begin{theorem}\label{thm:pseudo_imp_dist}
  If $\pi$ is the stationary distribution on a reversible, irreducible, finite Markov chain $\Omega,\mathbf{Q}$ with $G(1)=2$ then $(\Omega,\mathcal{P}(\Omega),\pi),\widehat{\mathbf{P}}_t$ is a pseudofinite Markov chain.
\end{theorem}
\begin{proof}
  \begin{itemize}
  \item (Stochasticity) Remembering that all integrals are with respect to $\pi$, for any $\omega$
\[\int\widehat{\mathbf{P}}_t(\omega,\omega')d\omega'
=\sum_{\omega'}\frac{\mathbf{P}_t(\omega,\omega')}{\pi(\omega')}\pi(\omega')
=\sum_{\omega'}\mathbf{P}_t(\omega,\omega')
=1
\]
by the stochasticity of $\mathbf{P}_t$.
\item (Symmetry) Since $\pi$ is a stationary distribution, we have
\[\widehat{\mathbf{P}}_t(\omega,\omega')=\frac{\mathbf{P}_t(\omega,\omega')}{\pi(\omega')}=\frac{\mathbf{P}_t(\omega',\omega)}{\pi(\omega)}=\widehat{\mathbf{P}}_t(\omega',\omega)\]
using the reversibility the original Markov chain.
\item (Chapman-Kolmogorov, both forms) 
\begin{align*}
\widehat{\mathbf{P}}_{s+t}(\omega,\omega')
&=\frac{\mathbf{P}_{s+t}(\omega,\omega')}{\pi(\omega')}\\
&=\frac{\sum_\xi\mathbf{P}_s(\omega,\xi)\mathbf{P}_t(\xi,\omega')}{\pi(\omega')}\\
&=\sum_\xi\frac{\mathbf{P}_s(\omega,\xi)\mathbf{P}_t(\xi,\omega')}{\pi(\xi)\pi(\omega')}\pi(\xi)\\
&=\int\widehat{\mathbf{P}}_s(\omega,\xi)\widehat{\mathbf{P}}_t(\xi,\omega')d\xi.
\end{align*}
\item (Boundedness) Trivial since $\int\widehat{\mathbf{P}}_t(\omega,\omega)d\omega$ is a finite sum in this case.
\item (Normality) By assumption we have
\[G(1)=\sum_\omega\mathbf{P}_1(\omega,\omega)=\sum_\omega\frac{\pi(\omega)}{\pi(\omega)}\mathbf{P}_1(\omega,\omega)
=\sum_\omega\widehat{\mathbf{P}}_1(\omega,\omega)\pi(\omega)=\int\widehat{\mathbf{P}}_1(\omega,\omega)d\omega=2.\]
\item (Continuity) Observe that
  \begin{align*}
    ||\widehat{\mathbf{P}}_t||_{L^2(\pi)}^2
=\sum_{\omega,\omega'}(\widehat{\mathbf{P}}_t(\omega,\omega'))^2\pi(\omega)\pi(\omega')\\ 
&=\sum_{\omega,\omega'}(\mathbf{P}_t(\omega,\omega'))^2\frac{\pi(\omega)}{\pi(\omega')}\\
&=\sqrt{\sum_{\omega,\omega'}(\mathbf{P}_t(\omega,\omega'))^2\frac{\pi(\omega)}{\pi(\omega')}}^2.
 \end{align*}
So $||\widehat{\mathbf{P}}_t||_{L^2(\pi)}$ is a matrix norm.  Therefore
  \begin{align*}
    ||\widehat{\mathbf{P}}_t-\widehat{\mathbf{P}}_s||_{L^2(\pi)}^2
&=||\mathbf{P}_t(\omega,\omega')-\mathbf{P}_s(\omega,\omega')||^2\\
&=||e^{(s+(t-s))\mathbf{Q}}-e^{s\mathbf{Q}}||^2\\
&\leq(t-s)^2||\mathbf{Q}||^2e^{4s||\mathbf{Q}||}.
  \end{align*}
This approaches $0$ as $t$ approaches $s$.
  \end{itemize}
\end{proof}

\begin{definition}
  If $\Omega,\mathbf{Q}$ is a reversible, irreducible, finite Markov chain with stationary distribution $\pi$ we define the density array corresponding to $\Omega,\mathbf{Q}$ to be the density array corresponding to $(\Omega,\mathcal{P}(\Omega),\pi),\widehat{\mathbf{P}}_t$.
\end{definition}

\subsection{Statement of Main Results}

We are now prepared to state our main results:
\begin{theorem}\label{thm:main}
  Suppose $\Omega_{(n)},\mathbf{Q}_{(n)}$ is a bounded sequence of Markov chains such that the corresponding density arrays $(\mathbf{X}_{(n),i,j}(t))_{i,j\in\mathbb{N}}$ converge in distribution.  Then there is a pseudofinite Markov chain $\Omega,\hat p_t$ such that the associated Markov chain $(\mathbf{X}_{*,i,j}(t))_{i,j\in\mathbb{N}}$ is the limit in distribution of the sequence $(\mathbf{X}_{(n),i,j}(t))_{i,j\in\mathbb{N}}$.
\end{theorem}

\begin{theorem}\label{thm:main2}
  Let $(\mathbf{X}_{*,i,j})$ be a density array.  There is a bounded sequence of finite Markov chains $\Omega_{(n)},\mathbf{Q}_{(n)}$ such that, taking $(\mathbf{X}_{(n),i,j})$ to be the density array corresponding to $\Omega_{(n)},\mathbf{Q}_{(n)}$, the density arrays $(\mathbf{X}_{(n),i,j})$ converge in distribution to $(\mathbf{X}_{*,i,j})$.
\end{theorem}

Together with Theorem \ref{thm:pseudo_imp_dist}, these theorems give the complete cycle of equivalences: every bound sequence of Markov chains has a subsequence with convergent density arrays whose limit corresponds to a pseudofinite Markov chain, every pseudofinite Markov chain has a corresponding density array, and every density array is the limit of some bound sequence of Markov chains.

We will prove Theorem \ref{thm:main} in Section \ref{sec:limits} and Theorem \ref{thm:main2} in Section \ref{sec:sampling}

We would also like to show that each density array corresponds to a unique pseudofinite Markov chain.  This is not true, but in Lemma \ref{thm:uniqueness} we will show that we can also choose $\Omega,\hat p_t$ to have two additional properties---\emph{twin-freeness} and \emph{saturation} (both these notions will be defined in that section)---and Theorem \ref{thm:uniqueness} shows that pseudofinite Markov chains with these additional properties are unique.

\section{Ultraproducts}\label{sec:ultraproducts}

Before proving that bounded sequences of Markov chains have corresponding limit objects, we recall some basic facts about our main technique, the ultraproduct construction.  Rather than reiterate that development here, we briefly describe the construction, then state a theorem which encapsulates all needed properties of the construction and refer the reader to \cite{goldbring:_approx_logic_measure} for a proof and a detailed exposition of the technique.

An \emph{filter} on $\mathbb{N}$ is a collection $\mathcal{U}$ of subsets of $\mathbb{N}$ such that $\emptyset\not\in\mathcal{U}$, $\mathbb{N}\in\mathcal{U}$, and $\mathcal{U}$ is upwards closed and closed under finite intersections.  A filter $\mathcal{U}$ is an \emph{ultrafilter} if for any $A\subseteq\mathbb{N}$, either $A\in\mathcal{U}$ or $(\mathbb{N}\setminus A)\in\mathcal{U}$.

Ultrafilters have the convenient property that if $(r_n)_{n\in\mathbb{N}}$ is a bounded sequence of reals, there is a set $A\in\mathcal{U}$ such that $\lim_{n\in A}r_n$ converges; moreover, the value of this limit is determined by $\mathcal{U}$ (because $\mathcal{U}$ is closed under intersections).  We write $\lim_{\mathcal{U}}r_n$ for this value.  The ultraproduct construction can be seen as a generalization of this idea: it is a construction that makes essentially arbitrary limits converge.

Given a sequence $\Omega^{(n)}$ of sets and an ultrafilter $\mathcal{U}$, we consider $\hat\Omega$, the collection of sequences $\langle \omega^{(n)}\rangle$ such that for each $n$, $\omega^{(n)}\in\Omega^{(n)}$.  We identify sequences $\omega\sim\omega'$ if $\{n\mid \omega^{(n)}=\omega'{}^{(n)}\}\in\mathcal{U}$.  We take our space to be the quotient $\Omega=\hat\Omega/\mathop{\sim}$.

If for each $n$ we have a subset $A^{(n)}\subseteq\Omega^{(n)}$, we can define $A=\lim\langle A^{(n)}\rangle$ to be those $\omega$ such that $\{n\mid \omega^{(n)}\in A^{(n)}\}\in\mathcal{U}$.  Subsets of this form are called \emph{internal}.

Given operations on each $\Omega^{(n)}$, we can generally lift them to $\Omega$ by considering what happens ``almost always''---that is, for a set of $n$ belonging to $\mathcal{U}$.  In particular, if for each $n$, $\pi^{(n)}$ is a probability measure on $\Omega^{(n)}$, we immediately obtain a finitely additive measure $\pi$ on the internal subsets of $\Omega$ by setting $\pi(A)=\lim_{\mathcal{U}}\pi^{(n)}(A^{(n)})$.  This extends to a measure---the \emph{Loeb measure}---on the $\sigma$-algebra generated by the internal sets.

With more effort, we can show that the $L^2$ (and more generally, $L^p$) spaces on $(\Omega,\pi)$ are, in a suitable sense, limits of the $L^2$ spaces on $(\Omega^{(n)},\pi^{(n)})$.  This is the content of the following result, which summarizes the results in \cite{goldbring:_approx_logic_measure} which will be needed in this paper.

\begin{theorem}\label{thm:ultraproduct}
  Let $\{(\Omega^{(n)},\pi^{(n)})\}$ be a sequence of finite probability spaces with $|\Omega^{(n)}|\rightarrow\infty$.  For each $i,n$, let $f_i^{(n)}$ be a function on $\Omega^{(n)}$ with $L^2$ norm bounded by $B_i$ (independently of $n$).  For any infinite set $S\subseteq\mathbb{N}$, there exist:
  \begin{itemize}
  \item A probability space $(\Omega,\mathcal{B},\pi)$, and
  \item For every sequence of sets $\langle A^{(n)}\rangle$ with each $A^{(n)}$ a subset of $\Omega^{(n)}$, a set $\lim\langle A^{(n)}\rangle=A\subseteq\Omega$ in $\mathcal{B}$,
  \item For each $i$, $L^2$ functions $f_i$ with $L^2$ norm bounded by $B_i$,
  \end{itemize}
so that:
\begin{itemize}
\item $\mathcal{B}$ is generated by sets of the form $\lim\langle A^{(n)}\rangle$,
\item The operation $\lim$ commutes with union, intersection, and complement, so $\lim\langle A^{(n)}\cap B^{(n)}\rangle=\lim\langle A^{(n)}\rangle\cap\lim\langle B^{(n)}\rangle$ and similarly for $\cup$ and complement,
\item Given a finite set $I$, finitely many sequences $\langle A^{(n)}_{1}\rangle,\ldots,\langle A^{(n)}_{r}\rangle$ with each $A_{j}^{(n)}\subseteq \Omega^{(n)}$ and setting $A_j=\lim\langle A_{j}^{(n)}\rangle$, there is a set $S'\subseteq S$ such that
\[\lim_{n\in S'}\int_{\bigcap_{j\leq r}A_{j}^{(n)}}\prod_{i\in I}f^{(n)}_i d\pi^{(n)}=\int_{\bigcap_{j\leq r}A_j}\prod_{i\in I}f_id\pi.\]
\end{itemize}
\end{theorem}

In particular, taking $I=\emptyset$, the last clause implies that
\[\lim_{n\in S'}\pi^{(n)}(\bigcap_{j\leq r}A_{j}^{(n)})=\pi(\bigcap_{j\leq r}A_j).\]

We call such a probability space $(\Omega,\mathcal{B},\pi)$ together with the operation $\lim$ an \emph{ultraproduct} of the sequence $\{\Omega^{(n)}\}$.  When we have specified a set $S$ in the theorem, we say the ultraproduct \emph{concentrates} on $S$.  The sets $\lim\langle A^{(n)}\rangle$ are called \emph{internal subsets} of $\Omega$.

\section{Limits of Bounded Sequences}\label{sec:limits}

In order to show convergence in distribution, we will need the following result:
\begin{lemma}\label{thm:convergence_in_dist}
Let $f:\mathbb{R}^k\rightarrow\mathbb{R}$ be a function bounded by $K$ such that whenever $|x-y|<\delta$, $|f(x)-f(y)|<\epsilon/2$.  Let $\mathbf{X}$ be a $\mathbb{R}^k$-valued random variable, let $\mathbf{A}_0,\ldots,\mathbf{A}_d$ be pairwise disjoint events and $I_1,\ldots, I_d$ subsets of $\mathbb{R}^k$ such that (writing $\mathbb{P}$ for the law of $\mathbf{X}$):
\begin{itemize}
\item $\mathbb{P}(\bigcup_{i\leq d}\mathbf{A}_i)=1$, %
\item $\mathbb{P}(\mathbf{X}\in I_i\mid \mathbf{A}_i)=1$ when $1\leq i\leq d$,
\item If $x,y\in I_i$ then $|x-y|<\delta$,
\item $\mathbb{P}(\mathbf{A}_0)<\epsilon/2K$.
\end{itemize}
For each $i\in[1, d]$, fix $r_i\in I_i$ and let $\alpha_i=\mathbb{P}(\mathbf{A}_i)$.  Then
\[\left|\mathbb{E}(f(\mathbf{X}))-\sum_{i\in[1, d]}\alpha_i r_i\right|<\epsilon.\]
\end{lemma}
\begin{proof}
  We have
\[\mathbb{E}(f(\mathbf{X}))=\mathbb{E}(f(\mathbf{X})\mid \mathbf{A}_0)\mathbb{P}(\mathbf{A}_0)+\sum_{i\in [1,d]}\mathbb{E}(f(\mathbf{X})\mid \mathbf{A}_i)\mathbb{P}(\mathbf{A}_i).\]
Since $f$ is bounded and $\mathbb{P}(\mathbf{A}_0)<\epsilon/2K$, we have 
\[\left|\mathbb{E}(f(\mathbf{X})\mid \mathbf{A}_0)\mathbb{P}(\mathbf{A}_0)\right|<K\cdot \epsilon/2K=\epsilon/2.\]
For each $i\leq d$, we have
\[\left|\mathbb{E}(f(\mathbf{X})\mid \mathbf{X}\in A_i)-r_i\right|<\epsilon/2.\]
Therefore
\[\left|\mathbb{E}(f(\mathbf{X}))-\sum_{i\in[1, d]}\alpha_i r_i\right|< \epsilon/2+\sum_{i\in[1, d]}\alpha_i\epsilon/2\leq\epsilon.\]
\end{proof}

A consequence is the following standard fact:
\begin{lemma}\label{thm:convergence_in_dist_for_l2}
  Let $f$ be a bounded continuous function.  Then for every $\epsilon>0$ there is a $\delta>0$ such that $||\mathbf{X}-\mathbf{Y}||_{L^2}<\delta$ implies that
\[\left|\mathbb{E}(f(\mathbf{X}))-\mathbb{E}(f(\mathbf{Y}))\right|<\epsilon.\]
\end{lemma}

\ 

Let $\Omega_{(n)},\mathbf{Q}_{(n)}$ be a bounded sequence of Markov chains with stationary distributions $\pi_{(n)}$.  Rather than directly taking an ultraproduct, we will pass to the eigenvector representation.  We may view each $\widehat{\mathbf{P}}_{(n),t}$ as an operator on the $L^2(\pi_{(n)})$-measurable functions by
\[(\widehat{\mathbf{P}}_{(n),t}f)(\omega)=\int \widehat{\mathbf{P}}_{(n),t}(\omega,\omega')f(\omega')d\omega'.\]
Since $\widehat{\mathbf{P}}_{(n),t}$ is symmetric, this operator is Hermitian.  Using Perron-Frobenius, we see that the $1$ is an eigenvalue with corresponding eigenvector $\omega\mapsto 1$ and that all other eigenvalues are in the range $[0,1)$.  It follows that if $1=\lambda_{(n),0}>\lambda_{(n),1}\geq\cdots$ are the eigenvalues of $\widehat{\mathbf{P}}_{(n),1}$ and $\nu_{(n),0},\ldots,$ are a corresponding orthonormal sequence of eigenvectors,
\[\widehat{\mathbf{P}}_{(n),t}(\omega,\omega')=\sum_i\lambda_{(n),i}^t\nu_{(n),i}(\omega)\nu_{(n),i}(\omega').\]
Note that $G_{(n)}(t)=\sum_i\lambda_{(n),i}^t$.

We apply Theorem \ref{thm:ultraproduct} to the functions $\nu_{(n),i}$ and the sequences $\lambda_{(n),i}$ (which we may view as constant functions).  We obtain a probability space $(\Omega,\mathcal{B},\pi)$, limiting values $\lambda_i\in[0,1]$ and measurable functions $\nu_i$ with $||\nu_i||\leq 1$.

\begin{lemma}
  $\lim_{i\rightarrow\infty}\lambda_i=0$.
\end{lemma}
\begin{proof}
  Suppose not.  Since $\lambda_i\geq 0$ for all $i$, it follows that $\lim_{i\rightarrow\infty}\sum_{j\leq i}\lambda_j=\infty$.  Pick $i$ large enough that $\sum_{j\leq i}\lambda_j>2$; then there must be an infinite set $S$ so that for $n\in S$, $\sum_{j\leq i}\lambda_{(n),j}>2$.  But this contradicts the fact that $\sum_i\lambda_{(n),i}=G_{(n)}(1)=2$.
\end{proof}

For any $i,j,n$ we have
\[\int \nu_{(n),i}(\omega)\nu_{(n),j}(\omega)d\omega=\left\{\begin{array}{ll}1&\text{if }i=j\\0&\text{otherwise}\end{array}\right.\]
so the $\nu_i$ have $L^2$ norm $1$ and are pairwise orthogonal.  Also, since $\nu_{(n),0}$ is constantly equal to $1$ for all $n$, $\nu_0$ is constantly equal to $1$.  Therefore for $i>0$, $\int \nu_i(\omega)d\omega=0$.

We define
\[\hat p_t(\omega,\omega')=\sum_i\lambda_i^t\nu_i(\omega)\nu_i(\omega').\]
We will show that $\hat p_t$ gives a pseudofinite Markov chain.  Symmetry of $\hat p_t(\omega,\omega')$ follows immediately from the definition.

\begin{lemma}[Stochasticity of $\hat p_t$]
  For every $t>0$ and almost every $\omega$, $\int\hat p_t(\omega,\omega')d\omega'=1$.
\end{lemma}
\begin{proof}
Let $t$ be given.  We will show that for every $\epsilon>0$, $||1-\int\hat p_t(\omega,\omega')d\omega'||_{L^2}<\epsilon$.

Fix $\epsilon>0$.  Pick $k$ large enough that $\sum_{i>k}\lambda_i^t<\epsilon$, so
\[||\hat p_t(\omega,\omega')-\sum_{i\leq k}\lambda_i^t\nu_i(\omega)\nu_i(\omega')||_{L^2}<\epsilon.\]

Since $\int\nu_i(\omega)d\omega=0$ for $i>0$, 
\[\int\sum_{i\leq k}\lambda_i^t\nu_i(\omega)\nu_i(\omega')d\pi(\omega')=
1+\sum_{0<i\leq k}\lambda_i^t\nu_i(\omega)\int\nu_i(\omega')d\pi(\omega')
=1.\]

Therefore
\begin{align*}
||1-\int\hat p_t(\omega,\omega')d\pi(\omega')||_{L^2}
&\leq ||1-\int\sum_{i\leq k}\lambda_i^t\nu_i(\omega)\nu_i(\omega')d\pi(\omega')||_{L^2}\\
&\ \ \ \ +||\int[\hat p_t(\omega,\omega')-\sum_{i\leq k}\lambda_i^t\nu_i(\omega)\nu_i(\omega')]d\pi(\omega')||_{L^2}\\
&\leq 0+||\hat p_t(\omega,\omega')-\sum_{i\leq k}\lambda_i^t\nu_i(\omega)\nu_i(\omega')||_{L^2}\\
&<\epsilon.
\end{align*}

Since this holds for every $\epsilon>0$, $||1-\int\hat p_t(\omega,\omega')d\omega'||_{L^2}=0$, so the set of $\omega$ such that $\int\hat p_t(\omega,\omega')d\omega'\neq 1$ must have measure $0$.
\end{proof}

\begin{lemma}[Chapman-Kolmogorov for $\hat p_t$]
  For any $s,t>0$ and almost every $\omega,\omega'$, $\hat p_{s+t}(\omega,\omega')=\int \hat p_s(\omega,\xi)\hat p_t(\xi,\omega')d\xi$.  Additionally, for almost every $\omega$, $\hat p_{s+t}(\omega,\omega)=\int \hat p_s(\omega,\xi)\hat p_t(\xi,\omega)d\xi$.
\end{lemma}
\begin{proof}
  The arguments for the two parts are identical except for notation.  We give the Chapman-Kolmogorov case; for the diagonal case, simply treat $\omega'$ as being equal to $\omega'$ (and consider the $L^1$ norm on $\Omega$ instead of $\Omega^2$).

  Let $s,t$ be given.  We will show that for every $\epsilon>0$, 
\[||\hat p_{s+t}(\omega,\omega')-\int \hat p_s(\omega,\xi)\hat p_t(\xi,\omega')d\xi||_{L^1}<\epsilon.\]

Fix $\epsilon>0$.  Pick $k$ large enough that $\sum_{i>k}\lambda_i^s<\epsilon/3G(2t)$, $\sum_{i>k}\lambda_i^t<\epsilon/3G(2s)$, and $\sum_{i>k}\lambda_i^{s+t}<\epsilon/3$.  Then
\[||\hat p_s(\omega,\xi)-\sum_{i\leq k}\lambda_i^s\nu_i(\omega)\nu_i(\xi)||_{L^2}<\epsilon/3G(2t)\]
and similarly for $\hat p_t$ and $\hat p_{s+t}$.  This in turn implies that
\begin{align*}
&  ||\int \left[\hat p_s(\omega,\xi)\hat p_t(\xi,\omega')- \left(\sum_{i\leq k}\lambda_i^s\nu_i(\omega)\nu_i(\xi)\right)\left(\sum_{j\leq k}\lambda_j^t\nu_j(\xi)\nu_j(\omega')\right)\right]d\xi||_{L^1}\\
\leq&||\hat p_s(\omega,\xi)\hat p_t(\xi,\omega')- \left(\sum_{i\leq k}\lambda_i^s\nu_i(\omega)\nu_i(\xi)\right)\left(\sum_{j\leq k}\lambda_j^t\nu_j(\xi)\nu_j(\omega')\right)||_{L^1}\\
\leq&||\left(\hat p_s(\omega,\xi)-\sum_{i\leq k}\lambda_i^s\nu_i(\omega)\nu_i(\xi)\right)\hat p_t(\xi,\omega')||_{L^1}\\
&\ \ \ \ +||\left(\sum_{i\leq k}\lambda_i^s\nu_i(\omega)\nu_i(\xi)\right)\left(\hat p_t(\xi,\omega')-\sum_{j\leq k}\lambda_j^t\nu_j(\xi)\nu_j(\omega')\right)||_{L^1}\\
\leq&||\hat p_s(\omega,\xi)-\sum_{i\leq k}\lambda_i^s\nu_i(\omega)\nu_i(\xi)||_{L^2}||\hat p_t(\xi,\omega')||_{L^2}\\
&\ \ \ \ +||\sum_{i\leq k}\lambda_i^s\nu_i(\omega)\nu_i(\xi)||_{L^2}||\left(\hat p_t(\xi,\omega')-\sum_{j\leq k}\lambda_j^t\nu_j(\xi)\nu_j(\omega')\right)||_{L^2}\\
\leq&\frac{\epsilon}{3G(2t)}G(2t)+G(2s)\frac{\epsilon}{3G(2s)}\\
&=\frac{2\epsilon}{3}.
\end{align*}

Observe that
\begin{align*}
\int \left(\sum_{i\leq k}\lambda_i^s\nu_i(\omega)\nu_i(\xi)\right)\left(\sum_{j\leq k}\lambda_j^t\nu_j(\xi)\nu_j(\omega')\right)d\xi
&=\sum_{i\leq k}\lambda_i^s\lambda_i^t\nu_i(\omega)\nu_i(\omega')\int\nu_i(\xi)^2d\xi\\
&\ \ \ \ \ \ \ +\sum_{i\neq j, i,j\leq k}\lambda_i^s\lambda_j^t\nu_i(\omega)\nu_j(\omega')\int\nu_i(\xi)\nu_j(\xi)d\xi\\
&=\sum_{i\leq k}\lambda_i^{s+t}\nu_i(\omega)\nu_i(\omega').\\
\end{align*}

Therefore as in the previous lemma,
\begin{align*}
  ||\hat p_{s+t}(\omega,\omega')-\int \hat p_s(\omega,\xi)\hat p_t(\xi,\omega')d\xi||_{L^1}
&\leq||\hat p_{s+t}(\omega,\omega')-\sum_{i\leq k}\lambda_i^{s+t}\nu_i(\omega)\nu_i(\omega')||_{L^1}\\
&\ \ \ \ \ \ \ \ +||\sum_{i\leq k}\lambda_i^{s+t}\nu_i(\omega)\nu_i(\omega')-\int \hat p_s(\omega,\xi)\hat p_t(\xi,\omega')d\xi||_{L^1}\\
&\leq \epsilon/3+2\epsilon/3\\
&=\epsilon.
\end{align*}
\end{proof}

\begin{lemma}[Boundedness and Normalization of $\hat p_t$]
  For every $t>0$, $\int\hat p_t(\omega,\omega)d\omega$ is finite and $\int\hat p_1(\omega,\omega)d\omega=2$.
\end{lemma}
\begin{proof}
We have
\[\int\hat p_t(\omega,\omega)d\omega=\sum_i\lambda_i^t\]
by definition.  Choosing $k$ large enough that $\sum_{i>k}\lambda_i^t<\epsilon$ we have
\[\int\hat p_t(\omega,\omega)d\omega=\sum_{i\leq k}\lambda_i^t+\sum_{i>k}\lambda_i^t<\sum_{i\leq k}\lambda_i^t+\epsilon.\]
Since for every $n$, $\sum_{i\leq k}\lambda_{(n),i}^t\leq\sum_i\lambda_{(n),i}^t\leq B_t$, we have 
\[\int \hat p_t(\omega,\omega)d\omega\leq B_t+\epsilon.\]

To prove equality for $t=1$ rather than a bound, we have to work a bit harder.  Choose $k$ large enough that $\lambda_k<\epsilon^2/2B^2_{1/2}$.  Consider any large enough $n$ such that $\lambda_{(n),k}\leq\epsilon^2/B^2_{1/2}$.  We have $\sum_{i>k}\lambda_{(n),i}^{1/2}\leq\sum_i\lambda_{(n),i}^{1/2}\leq B_{1/2}$ and for each $i>k$, $\lambda_{(n),i}^{1/2}\leq\epsilon/B_{1/2}$.  Among all possible values for the $\lambda_{(n),i}^{1/2}$, we maximize $\sum_{i>k}(\lambda^{1/2}_{(n),i})^2$ by choosing as many as possible as large as possible and equal, and the rest $0$.  But at this maximum, the first $B^2_{1/2}/\epsilon$ values of $\lambda_{(n),i}$ are each $\epsilon/B_{1/2}$, and the remainder are $0$; this means that $\sum_{i>k}(\lambda^{1/2}_{(n),i})^2=\sum_{i>k}\lambda_{(n),i}\leq (B^2_{1/2}/\epsilon)(\epsilon/B_{1/2})^2=\epsilon$.  Since $\sum_i\lambda_{(n),i}=2$, it follows that $2-\sum_{i\leq k}\lambda_{(n),i}\leq\epsilon$.  Since this holds for almost every $n$, we have $2-\sum_{i\leq k}\lambda_i\leq\epsilon$.  Since this holds for every $\epsilon>0$, $\sum_{i}\lambda_i=2$.
\end{proof}

\begin{lemma}
  For every $t>0$, $\lim_{s\rightarrow t}||\hat p_t-\hat p_s||_{L^2}=0$.
\end{lemma}
\begin{proof}
  Immediate from the definition of $\hat p_t$ and Lemma \ref{thm:close_all}.
\end{proof}

Putting these together, we see that $(\Omega,\mathcal{B},\pi),\hat p_t$ is a pseudofinite Markov chain.  For each $n$, $(\mathbf{X}_{(n),i,j})_{i,j\in\mathbb{N}}$ be the density array corresponding to $\Omega_{(n)},\mathbf{Q}_{(n)}$ and let $(\mathbf{X}_{*,i,j})_{i,j\in\mathbb{N}}$ be the density array corresponding to $(\Omega,\mathcal{B},\pi),\hat p_t$.  We now turn to showing that $(\mathbf{X}_{*,i,j})_{i,j\in\mathbb{N}}$ is the limit of a convergent subsequence of $(\mathbf{X}_{(n),i,j})_{i,j\in\mathbb{N}}$.

\begin{theorem}
Let $f:\mathbb{R}^{m+1}\rightarrow\mathbb{R}$ be bounded and continuous.  Let $i_0,\ldots,i_m, j_0,\ldots, j_m, t_0,\ldots t_m$ be given.  Then there is an infinite set $S$ such that 
\[\lim_{n\in S}\mathbb{E}(f(\mathbf{X}_{(n),i_0,j_0}(t_0),\ldots,\mathbf{X}_{(n),i_m,j_m}(t_m)))=\mathbb{E}(f(\mathbf{X}_{*,i_0,j_0}(t_0),\ldots,\mathbf{X}_{*,i_m,j_m}(t_m))).\]
\end{theorem}
\begin{proof}

To simplify notation, write $\mathbf{X}_{(n)}$ for the $\mathbb{R}^{m+1}$-valued random variable
\[(\mathbf{X}_{(n),i_0,j_0}(t_0),\ldots,\mathbf{X}_{(n),i_m,j_m}(t_m))\]
and $\mathbf{X}_{*}$ for the $\mathbb{R}^{m+1}$-valued random variable
\[(\mathbf{X}_{*,i_0,j_0}(t_0),\ldots,\mathbf{X}_{*,i_m,j_m}(t_m)).\]

Let $M\geq i_r,j_r$ for all $r\leq m$.  Let $K$ be the bound on $f$ and fix $\epsilon>0$.  Fix a compact subset $C$ of $\mathbb{R}^{m+1}$ so that $\mathbb{P}(\mathbf{X}_*\not\in C)<\epsilon/4K$.  Let $\delta$ be small enough that $|x-y|<\delta$ and $x,y\in C$ implies $|f(x)-f(y)|<\epsilon/2$.

For each $k$, write $\mathbf{Y}_{*,k}$ 
for the random variable given by choosing $\omega_{i_r},\omega_{j_r}$ randomly according to $\pi$
and taking the value
\[(\sum_{i\leq k}\lambda_{i}^{t_0}\nu_{i}(\omega_{i_0})\nu_{i}(\omega_{j_0}),\ldots,\sum_{i\leq k}\lambda_{i}^{t_m}\nu_{i}(\omega_{i_m})\nu_{i}(\omega_{j_m})).\]
That is $\mathbf{Y}_{*,k}$ is the approximation to $\mathbf{X}_*$ using only eigenvectors up to $k$.  Let $\mathbf{Y}_{(n),k}$ be given analogously by
\[(\sum_{i\leq k}\lambda_{(n),i}^{t_0}\nu_{(n),i}(\omega_{i_0})\nu_{(n),i}(\omega_{j_0}),\ldots,\sum_{i\leq k}\lambda_{(n),i}^{t_m}\nu_{(n),i}(\omega_{i_m})\nu_{(n),i}(\omega_{j_m})).\]

Then $\mathbf{Y}_{*,k}$ $L^2$ converges to $\mathbf{X}_{*}$, so in particular we may fix a value of $k$ large enough that 
\[|\mathbb{E}(f(\mathbf{Y}_{*,k}))-\mathbb{E}(f(\mathbf{X}_*))|<\epsilon/2.\]

It will suffice to show that we can find an infinite set $S$ such that for $n\in S$,
 \[|\mathbb{E}(f(\mathbf{Y}_{(n),k}))-\mathbb{E}(f(\mathbf{X}_{(n)}))|<\epsilon\]
and
\[|\mathbb{E}(f(\mathbf{Y}_{*,k}))-\mathbb{E}(f(\mathbf{Y}_{(n),k}))|<\epsilon.\]

Cover $C$ with finitely many pairwise disjoint sets $I_1,\ldots,I_d$ so that $x,y\in I_i$ implies $|x-y|<\delta/2$.  For each $n$, let $\mathbf{A}_{(n),i}$ be the event that $\mathbf{Y}_{(n),k}\in I_i$ and $\mathbf{A}_{(n),0}$ be the event that $\mathbf{Y}_{(n),k}\not\in\bigcup_{i\leq d}I_i$.  Let $\mathbf{A}_i=\lim\langle \mathbf{A}_{(n),i}\rangle$.  Note that the event $\mathbf{A}_i$ for $i\in[1,d]$ implies that $\mathbf{Y}_{*,k}$ belongs to the closure of $I_i$.

Fix $r_i\in I_i$ for each $i$, let $r=\max_i |r_i|$, and let $\alpha_i=\mathbb{P}(\mathbf{A}_i)$.  There is an infinite set $S$ such that for each $n\in S$:
\begin{itemize}
\item  $|\mathbb{E}(f(\mathbf{Y}_{(n),k}))-\mathbb{E}(f(\mathbf{X}_{(n)}))|<\epsilon$,
\item For each $i\in[1,d]$, $|\mathbb{P}(\mathbf{A}_{(n),i})-\alpha_i|<\epsilon/2r_i$,
\item $|\mathbb{P}(\mathbf{A}_{(n),0})-\alpha_0|<\epsilon/4K$.
\end{itemize}

Consider any $n\in S$.  Since $\alpha_0<\epsilon/4K$, we have $\mathbb{P}(\mathbf{A}_{(n),0})<\epsilon/2K$.  Then by Lemma \ref{thm:convergence_in_dist},
\[|\mathbb{E}(f(\mathbf{Y}_{(n),k}))-\sum_{i\in[1,d]}\mathbb{P}(\mathbf{A}_{(n),i})r_i|<\epsilon\]
and since
\[|\sum_{i\in[1,d]}\mathbb{P}(\mathbf{A}_{(n),i})r_i-\sum_{i\in[1,d]}\alpha_i r_i|\leq\epsilon/2\]
we have
\[|\mathbb{E}(f(\mathbf{Y}_{(n),k}))-\sum_{i\in[1,d]}\alpha_i r_i|<3\epsilon/2.\]
On the other hand, by Lemma \ref{thm:convergence_in_dist} again,
\[|\mathbb{E}(f(\mathbf{Y}_{*,k}))-\sum_{i\in[1,d]}\alpha_ir_i|<\epsilon/2\]
and therefore 
\[|\mathbb{E}(f(\mathbf{Y}_{*,k}))-\mathbb{E}(f(\mathbf{Y}_{(n),k}))|<2\epsilon.\]

It follows that
\[|\mathbb{E}(f(\mathbf{X}_{*}))-\mathbb{E}(f(\mathbf{X}_{(n)}))|<7\epsilon/2.\]

Since we can find infinitely many such $n$ for each $\epsilon>0$, we may take a countable sequence $\epsilon_1>\epsilon_2>\cdots$ and corresponding $n_1<n_2<\cdots$, and then the set $S=\{n_1,n_2,\ldots\}$ has the property that
\[\lim_{n\in S}\mathbb{E}(f(\mathbf{X}_{(n)}))=\mathbb{E}(f(\mathbf{X}_*)).\]
\end{proof}

In particular, if the sequence $(\mathbf{X}_{(n),i,j})$ converges in distribution, for each bounded continuous $f$ and any infinite $S$,
\[\lim_{n}\mathbb{E}(f(\mathbf{X}_{(n),i_0,j_0}(t_0),\ldots,\mathbf{X}_{(n),i_m,j_m}(t_m)))=\lim_{n\in S}\mathbb{E}(f(\mathbf{X}_{(n),i_0,j_0}(t_0),\ldots,\mathbf{X}_{(n),i_m,j_m}(t_m)))\]
and so
\[\lim_{n}\mathbb{E}(f(\mathbf{X}_{(n),i_0,j_0}(t_0),\ldots,\mathbf{X}_{(n),i_m,j_m}(t_m)))=\mathbb{E}(f(\mathbf{X}_{*,i_0,j_0}(t_0),\ldots,\mathbf{X}_{*,i_m,j_m}(t_m))).\]
Therefore the sequence $(\mathbf{X}_{(n),i,j})$ converges in distribution to $(\mathbf{X}_{*,i,j})$.  This proves Theorem \ref{thm:main}.

\section{Sampling}\label{sec:sampling}

\begin{reptheorem}{thm:main2}
Let $(\mathbf{X}_{*,i,j})$ be a density array.  There is a bounded sequence of finite Markov chains $\Omega_{(n)},\mathbf{Q}_{(n)}$ such that, taking $(\mathbf{X}_{(n),i,j})$ to be the density array corresponding to $\Omega_{(n)},\mathbf{Q}_{(n)}$, the density arrays $(\mathbf{X}_{(n),i,j})$ converge in distribution to $(\mathbf{X}_{*,i,j})$.
\end{reptheorem}
\begin{proof}
We construct finite Markov chains $\Omega_{(n)},\mathbf{Q}_{(n)}$ by taking $\Omega_{(n)}$ to be an $n$ point set $\{\omega_0,\ldots,\omega_{n-1}\}$ of $\Omega$ and setting the values $\hat p^{(n)}_t(\omega_i,\omega_j)$ according to $\mathbf{X}_{*,i,j}(t)$.  When $n$ is large, $\Omega_{(n)},\widehat{p}_t^{(n)}$ will with high probability almost define a finite Markov chain, but there my be some error---for instance, the rows will sum to a number near $1$, but not to exactly $1$.

Nonetheless, we can let $(\mathbf{Y}_{(n),i,j})$ be the corresponding random variables given by choosing $\xi_1,\xi_2,\ldots$ from $\Omega_{(n)}$ uniformly at random and taking $\mathbf{Y}_{(n),i,j}(t)=\hat p^{(n)}_t(\xi_i,\xi_j)$.  
This selection induces a random function $\boldsymbol{\rho}_n:\mathbb{N}\rightarrow\mathbb{N}$ so $\xi_i=\omega_{\boldsymbol{\rho}(i)}$, and so in particular $\mathbf{Y}_{(n),i,j}(t)=\mathbf{X}_{*,\boldsymbol{\rho}_n(i),\boldsymbol{\rho}_n(j)}(t)$.

Fix a bounded and continuous function $f:\mathbb{R}^{m+1}\rightarrow\mathbb{R}$ and values $i_0,\ldots,i_m,j_0,\ldots,j_m,t_0,\ldots,t_m$, and consider the sequence of values
\[f(\mathbf{Y}_{(n),i_0,j_0}(t_0),\ldots,\mathbf{Y}_{(n),i_m,j_m}(t_m)).\]
Note that $\mathbb{E}(f(\mathbf{Y}_{(n),i_0,j_0}(t_0),\ldots,\mathbf{Y}_{(n),i_m,j_m}(t_m)))$ is the average over all choices of $\boldsymbol{\rho}_n$,
\[\mathbb{E}(f(\mathbf{Y}_{(n),i_0,j_0}(t_0),\ldots,\mathbf{Y}_{(n),i_m,j_m}(t_m)))=\frac{1}{r}\sum_{\boldsymbol{\rho}_n}f(\mathbf{X}_{*,\boldsymbol{\rho}_n(i_0),\boldsymbol{\rho}_n(j_0)}(t_0),\ldots,\mathbf{X}_{*,\boldsymbol{\rho}_n(i_m),\boldsymbol{\rho}_n(j_m)}(t_m)).\]
When $n$ is much larger than $m$, almost every choice of $\boldsymbol{\rho}_n$ is injective on the elements $\{i_0,\ldots,i_m,j_0,\ldots,j_m,t_0,\ldots,t_m\}$, so
\[\mathbb{E}(f(\mathbf{X}_{*,\boldsymbol{\rho}_n(i_0),\boldsymbol{\rho}_n(j_0)}(t_0),\ldots,\mathbf{X}_{*,\boldsymbol{\rho}_n(i_m),\boldsymbol{\rho}_n(j_m)}(t_m)))=\mathbb{E}(f(\mathbf{X}_{*,i_0,j_0}(t_0),\ldots,\mathbf{X}_{*,i_m,j_m}(t_m))).\]
Since $(\mathbf{X}_{*,i,j})$ is dissociated, the value of $f(\mathbf{X}_{*,\boldsymbol{\rho}_n(i_0),\boldsymbol{\rho}_n(j_0)}(t_0),\ldots,\mathbf{X}_{*,\boldsymbol{\rho}_n(i_m),\boldsymbol{\rho}_n(j_m)}(t_m))$ for two different values of $\boldsymbol{\rho}_n$ are independent if the images of $\{i_0,\ldots,i_m,j_0,\ldots,j_m,t_0,\ldots,t_m\}$ are disjoint.  Therefore a standard Azuma's inequality argument shows that with probability exponentially approaching $1$,
\[\mathbb{E}(f(\mathbf{Y}_{(n),i_0,j_0}(t_0),\ldots,\mathbf{Y}_{(n),i_m,j_m}(t_m)))=\mathbb{E}(f(\mathbf{X}_{*,i_0,j_0}(t_0),\ldots,\mathbf{X}_{*,i_m,j_m}(t_m)).\]
This shows that the countable matrix of random variables $(\mathbf{Y}_{(n),i,j}(q))$ with $q$ rational converges in distribution to $(\mathbf{X}_{*,i,j}(q))$.  However $\hat p^{(n)}_t$ is not exactly a finite Markov chain (it need only be ``nearly stochastic'' for instance), so it remains to find a finite Markov chain close to $\hat p^{(n)}_t$.  There is a natural choice---treat $\pi_{(n)}(\omega)=\sum_\xi\hat p^{(n)}_1(\omega,\xi)$ as the weight given to $\omega$, and normalize by multiplying by $\frac{1}{\pi_{(n)}(\omega)}$.  Since this should, for most $\omega$, be very close $n$, it should not be surprising that this normalization does not change the limiting distribution.  Checking this, especially checking that it does change the distribution on large products of the matrix, occupies the remainder of the proof.

By the Chapman-Kolmogorov property and exchangeability, with probability $1$, for any $i,j$ we have
\[\mathbf{X}_{*,i,j}(2)=\lim_{n\rightarrow\infty}\frac{1}{n}\sum_{k<n}\mathbf{X}_{*,i,k}(1)\mathbf{X}_{*,k,j}(1).\]
Corresponding to this, when $n$ is sufficiently large, with very high probability we have
\[\hat p^{(n)}_2(\omega,\omega')\approx\frac{1}{n}\sum_{\xi\in\Omega}\hat p^{(n)}_1(\omega,\xi)\hat p^{(n)}_1(\xi,\omega').\]
More strongly, we can do the following.  For any $S\subseteq\mathbb{N}$, define
$\delta(S)=\lim_{n\rightarrow\infty}\frac{|S\cap [0,n]|}{n}$
if this exists.  When $S$ is defined in a suitably exchangeable way (as all sets $S$ we consider will be), $\delta(S)$ exists with probability $1$.  Observe that since the second moment of $\mathbf{X}_{*,i,j}(1)$ exists (by the same argument as for pseudofinite Markov chains), let $I_{>B}=\{k\mid \mathbf{X}_{*,i,k}(1)\mathbf{X}_{*,k,j}(1)>B\}$ for $B$ sufficiently big, and with probability $1$, $\delta(I_{>B})<\epsilon$ (where $\epsilon$ depends on $B$).  Consider, for $0\leq c\leq \lceil B/\epsilon\rceil$, $I_c=\{k\mid \mathbf{X}_{*,i,k}(1)\mathbf{X}_{*,k,j}(1)\in [c\epsilon,(c+1)\epsilon)\}$.  When $\epsilon$ is small, 
\[\mathbf{X}_{*,i,j}(2)\approx\sum_{c}c\epsilon \delta(I_c).\]
Further we can define $I'_c=\{\xi\in\Omega\mid |\hat p^{(n)}_1(\omega,\xi)\hat p^{(n)}_1(\xi,\omega')-\delta|<\epsilon\}$, and with probability exponentially approaching $1$, a standard Chernoff argument ensures that when $n$ is large, $\mu(I'_c)$ is very close to $\delta(I_c)$ for all $\omega,\omega'$ simultaneously.  In particular, if $E\subseteq\Omega$ with $|E|/n$ small enough (depending on $B,\epsilon$), we still have
\[\hat p^{(n)}_2(\omega,\omega')\approx\frac{1}{n}\sum_{\xi\in\Omega\setminus E}\hat p^{(n)}_1(\omega,\xi)\hat p^{(n)}_1(\xi,\omega').\]

With probability $1$, $\hat p^{(n)}_1$ is a symmetric matrix.  Set $\pi_{(n)}(\omega)=\sum_{\omega'\in\Omega_{(n)}}\hat p^{(n)}_1(\omega,\omega')$.  We set $\mathbf{P}_{(n),1}(\omega,\omega')=\frac{\hat p^{(n)}_1(\omega,\omega')}{\pi_{(n)}(\omega)}$.  Then $\mathbf{P}_{(n),1}$ is a positive definite stochastic matrix and we can take $\mathbf{Q}_{(n)}$ to be its matrix logarithm.  

Let $\gamma=\sum_\omega\pi_{(n)}(\omega)$.  (With high probability, $\gamma$ is close to $n^2$.)  Then $\frac{1}{\gamma}\pi_{(n)}$ is the stationary distribution of $\mathbf{Q}_{(n)}$, so $\widehat{\mathbf{P}}_{(n),1}(\omega,\omega')=\frac{\gamma\hat p^{(n)}_1(\omega,\omega')}{\pi_{(n)}(\omega)\pi_{(n)}(\omega')}$.  It remains to show that the densities arrays corresponding to $\Omega_{(n)},\mathbf{Q}_{(n)}$ converge in distribution to $(\mathbf{X}_{*,i,j})$.

Let $E_\epsilon$ be those $\xi$ such that $|\pi_{(n)}(\xi)-n|\geq\xi$.  By the stochasticity of $(\mathbf{X}_{*,i,j})$, with high probability the set of $E_\epsilon$ has size $<\epsilon n$.  Therefore for any $\omega,\omega'$,
\begin{align*}
\widehat{\mathbf{P}}_{(n),2}(\omega,\omega')
&=\sum_{\xi}\widehat{\mathbf{P}}_{(n),1}(\omega,\xi)\widehat{\mathbf{P}}_{(n),1}(\xi,\omega')\frac{1}{\gamma}\pi_{(n)}(\xi)\\
&=\sum_\xi\frac{\gamma\hat p^{(n)}_1(\omega,\xi)\hat p^{(n)}_1(\xi,\omega')}{\pi_{(n)}(\omega)\pi_{(n)}(\xi)\pi_{(n)}(\omega')}\\
&=\frac{\gamma}{\pi_{(n)}(\omega)\pi_{(n)}(\omega')}\sum_\xi\frac{\hat p^{(n)}_1(\omega,\xi)\hat p^{(n)}_1(\xi,\omega')}{\pi_{(n)}(\xi)}\\
&\approx\frac{\gamma}{\pi_{(n)}(\omega)\pi_{(n)}(\omega')}\sum_{\xi\in\Omega\setminus E_\epsilon}\frac{\hat p^{(n)}_1(\omega,\xi)\hat p^{(n)}_1(\xi,\omega')}{n}\\
&\approx\frac{\gamma \hat p^{(n)}_2(\omega,\omega')}{\pi_{(n)}(\omega)\pi_{(n)}(\omega')}.
\end{align*}

We can iterate the same argument and conclude that for each $k$, when $n$ is sufficiently large (depending on $k$), with very high probability we have
\[\widehat{\mathbf{P}}_{(n),k}(\omega,\omega')\approx \frac{\gamma\hat p^{(n)}_k(\omega,\omega')}{\pi_{(n)}(\omega)\pi_{(n)}(\omega')}.\] 
Similarly, $\widehat{\mathbf{P}}_{(n),1}(\omega,\omega')$ is close to $\frac{\gamma}{\pi_{(n)}(\omega)\pi_{(n)}(\omega')}\sum_\xi\hat p^{(n)}_{1/2}(\omega,\xi)\hat p^{(n)}_{1/2}(\xi,\omega')$, so by the continuity of the matrix square root, $\frac{\gamma\hat p^{(n)}_{1/2}(\omega,\omega')}{\pi_{(n)}(\omega)\pi_{(n)}(\omega')}$ is close to $\widehat{\mathbf{P}}_{(n),1/2}(\omega,\omega')$.  Iterating and combining these two arguments, we get that $\frac{\hat p^{(n)}_{t}(\omega,\omega)}{\pi_{(n)}(\omega,\omega')}\approx\widehat{\mathbf{P}}_{(n),t}(\omega,\omega')$ for all rational $t$, and therefore with probability $1$, for each $t$, $\hat p^{(n)}_t-\widehat{\mathbf{P}}_{(n),t}$ converges to $0$ as $n$ gets large.

Therefore for rational $t$ the sequence of random variables $\mathbf{Y}_{(n),i,j}(t)-\mathbf{X}_{(n),i,j}(t)$ converges in $L^2$ (and even $L^\infty$) norm to $0$, so also in distribution to $0$, and therefore $\mathbf{X}_{(n),i,j}(t)$ converges in distribution to the same random variable as $\mathbf{Y}_{(n),i,j}(t)$, namely $\mathbf{X}_{*,i,j}(t)$.

To obtain convergence for all irrational $s$ simultaneously, fix an $s$, a bounded continuous $f$ and an $\epsilon>0$.  To simplify notation, we assume $f:\mathbb{R}\rightarrow\mathbb{R}$; the general case follows by the same argument.  Let $i,j$ be given.  There is a $\delta>0$ so that if $||\mathbf{X}-\mathbf{X}'||_{L^2}<\delta$ then $\left|\mathbb{E}(f(\mathbf{X}))-\mathbb{E}(f(\mathbf{X}'))\right|<\epsilon/3$.  We may choose a rational $t$ close enough to $s$ that $||\mathbf{X}_{*,i,j}(t)-\mathbf{X}_{*,i,j}(s)||_{L^2}<\delta$ and also for sufficiently large $n$
\[||\widehat{\mathbf{P}}_{(n),t}-\widehat{\mathbf{P}}_{(n),s}||_{L^2}=
||e^{t\mathbf{Q}_{(n)}}-e^{s\mathbf{Q}_{(n)}}||_{L^2}
\leq (s-t)||\mathbf{Q}_{(n)}||_{L^2}||\widehat{\mathbf{P}}_{(n),t}||_{L^2}||\widehat{\mathbf{P}}_{(n),s}||_{L^2}
<\delta.\]
Therefore when $n$ is sufficiently large we have
\begin{align*}
  \left|\mathbb{E}(f(\mathbf{X}_{*,i,j}(s)))-\mathbb{E}(f(\mathbf{X}_{(n),i,j}(s)))\right|
&\leq\left|\mathbb{E}(f(\mathbf{X}_{*,i,j}(s)))-\mathbb{E}(f(\mathbf{X}_{*,i,j}(t)))\right|\\
&\ \ \ \ +\left|\mathbb{E}(f(\mathbf{X}_{*,i,j}(t)))-\mathbb{E}(f(\mathbf{X}_{(n),i,j}(t)))\right|\\
&\ \ \ \ +\left|\mathbb{E}(f(\mathbf{X}_{(n),i,j}(t)))-\mathbb{E}(f(\mathbf{X}_{(n),i,j}(s)))\right|\\
&<\epsilon/3+\epsilon/3+\epsilon/3\\
&=\epsilon.
\end{align*}
\end{proof}

\section{Uniqueness}

We would like to additionally have the property that a pseudofinite Markov chain is determined by its density array: that two pseudofinite Markov chains with the same density array are actually isomorphic.  This is not the case, as an easy example shows.

Consider the pseudofinite Markov chain where $\Omega=\{a,b\}$ with $\pi(\{a\})=2/3$ and $\pi(\{b\})=1/3$ so the eigenvectors are the constant, the eigenvector $\nu_1(a)=\sqrt{2}/2$ and $\nu_1(b)=-\sqrt{2}$ with eigenvalue $\lambda_1=1/2$, and the eigenvector $\nu_2=-\nu_1$ with eigenvalue $\lambda_2=1/2$.  (That is, an arbitrary chain with two points with different measures.)

 We could modify this chain by replacing $a$ with an uncountable ``blob'' of points which mix very rapidly: $\Omega'=A\cup\{b\}$ where $(A,\lambda)$ is some probability measure space with $A$ uncountable, $\pi'(S)=(2/3)\lambda(S)$ when $S\subseteq A$, $\pi'(\{b\})=1/3$, and the eigenvectors are the constant, $\nu'_1(a)=\sqrt{2}/2$ for $a\in A$ and $\nu'_1(b)=-\sqrt{2}$ with eigenvalue $\lambda_1=1/2$, and $\nu'_2=-\nu'_1$ with eigenvalue $\lambda_2=1/2$.  These chains have the same density array, but are clearly not isomorphic.  Moreover, by blowing up $b$ instead of $A$---that is, taking $\Omega''=\{a\}\cup B$ and defining $\pi'',\nu''_1$ analogously---we can get two chains with the same density array such that neither embeds in the other.

The solution is take pseudofinite Markov chains with some additional properties to be canonical representatives of each density array.  (Our approach is similar to that taken in \cite{MR2594615} for graph limits.)

\begin{definition}
  Let $(\Omega,\mathcal{B},\pi),\hat p_t$ be a pseudofinite Markov chain with $\hat p_t(\omega,\omega')=\sum_i\lambda_i^t\nu_i(\omega)\nu_i(\omega')$.  A \emph{potential type} is a function $q:\mathbb{N}\rightarrow\mathbb{R}$.  We write $tp(\omega)$, the \emph{type} of $\omega$, for the function $tp(\omega)(i)=\nu_i(\omega)$.  We say $\omega$ \emph{realizes} $q$ if $q(i)=\nu_i(\omega)$ for all $i$.

  For $I\subseteq\mathbb{N}$ finite and $\epsilon>0$, the points which \emph{$I,\epsilon$-almost realize} $q$ are those $\omega\in\Omega$ such that for every $i\in I$, $|q(i)-\nu_i(\omega)|<\epsilon$.  $q$ is a \emph{wide type} if for every $I,\epsilon$, the set of $\omega$ which $I,\epsilon$-almost realize $q$ has positive measure.

  We say $(\Omega,\mathcal{B},\pi),\hat p_t$ is \emph{saturated} if whenever $q$ is a wide type, there is an $\omega\in\Omega$ realizing $q$.  We say $(\Omega,\mathcal{B},\pi),\hat p_t$ is \emph{twin-free} if whenever $q$ is a type, there is at most one $\omega\in\Omega$ realizing $q$.
\end{definition}

We say $\omega\in \Omega$ has wide type exactly if $tp(\omega)$ is a wide type.

Our definition of type is motivated by the model theoretic notion of the same name (and indeed, a type in our sense would be a partial type if we represented our Markov chains as first-order structures in an appropriate language).  Informally, a wide type is a point which ``ought to'' exist, in these sense that it is a limit of many points which are present.  Saturation says that all the points which should exist are actually present, and twin-freeness says that we don't have multiple points which are indistinguishable.


We ignore the behavior of non-wide types because the points with non-wide type are negligible.

\begin{lemma}
  Let $(\Omega,\mathcal{B},\pi),\hat p_t$ be a pseudofinite Markov chain.  Then almost every $\omega\in\Omega$, $tp(\omega)$ is wide.
\end{lemma}
\begin{proof}
For any finite set $I$ and any collection $\{Q_i\}_{i\in I}$ where each $Q_i$ is an interval in $\mathbb{R}$ with rational endpoints, let
\[D_{I,\{Q_i\}}=\{\omega\mid\forall i\in I\ \nu_i(\omega)\in Q_i\}.\]
There are countably many such $D_{I,\{Q_i\}}$, so it suffices to show that if $\omega$ fails to have wide type then it is contained in some $D_{I,\{Q_i\}}$ with measure $0$.

  Suppose $\omega$ does not have wide type.  Then there is an $\epsilon>0$ and a finite set $I$ such that the set of $\omega'$ so that, for every $i\in I$, $|\nu_i(\omega)-\nu_i(\omega')|<\epsilon$, has measure $0$.  For each $i\in I$, let $Q_i$ be some interval $(p_i,q_i)$ with $\nu_i(\omega)-\epsilon\leq p_i<\nu_i(\omega)<q_i\leq\nu_i(\omega)+\epsilon$ with $p_i$ and $q_i$ both rational; such an interval always exists.  Then $\omega\in D_{I,\{Q_i\}}$ and any $\omega'\in D_{I,\{Q_i\}}$ would $I,\epsilon$-almost realize $tp(\omega)$, and therefore $\mu(D_{I,\{Q_i\}})=0$.
\end{proof}

\begin{lemma}\label{thm:existence}
    Let $(\Omega,\mathcal{B},\pi),\hat p_t$ be a pseudofinite Markov chain.  Then there is a twin-free saturated pseudofinite Markov chain $(\Omega',\mathcal{B}',\pi'),\hat p'_t$ and a measurable, measure-preserving map $\rho:\Omega\rightarrow\Omega'$ so that $\hat p_t(\omega,\omega')=\hat p'_t(\rho(\omega),\rho(\omega'))$ for all $\omega,\omega'\in\Omega$.
\end{lemma}
\begin{proof}
  We take $\Omega'$ to be the set of types of $(\Omega,\mathcal{B},\pi),\hat p_t$.  We define eigenvectors $\nu'_i(q)=q(i)$ and set $\hat p'_t(q,q')=\sum_i\lambda_i^t\nu'_i(q)\nu'_i(q')$.  The measurable sets $\mathcal{B}'$ are generated by the level sets of $\nu'_i$; in particular, this ensures that each $\nu'_i$, and therefore $\hat p'_t$, is measurable.  We set $\pi'(\{q\mid \nu'_i(q)\leq c\})=\pi(\{\omega\mid\nu_i(\omega)\leq c\})$.

  We define $\rho(\omega)=tp(\omega)$, which ensures $\hat p_t(\omega,\omega')=\hat p'_t(\rho(\omega),\rho(\omega'))$.  The definition of $\pi'$ ensures that this is measurable and measure-preserving.

  Note that $tp(q)=q$, so $\Omega'$ has exactly one element of each type, so is certainly twin-free and saturated.
\end{proof}



Given two twin-free pseudofinite Markov chains $\hat p$ and $\hat p'$ with the same density array, we would like to simply match up points which have the same type.  However to do this, we need to show that the eigenvectors have the same distributions, and the only way we know of to do this is to show that $\hat p_1$ and $\hat p'_1$ can both embed in the same chain $\hat p_*$.  We therefore invoke \cite{MR2594615} to do precisely that.  The resulting approach is overkill---we are first showing that $\hat p_1$ and $\hat p'_1$ can embed in a common object, and then starting over to show isomorphism of $\hat p_t$ and $\hat p'_t$.  The alternative---repeating much of the proof from \cite{MR2594615} to get the result in one step---is not much simpler, and we expect that there is a direct proof that the eigenvectors of $\hat p_1$ and $\hat p'_1$ have the same distributions, which would give a simpler proof of the whole result.

  \begin{lemma}
     Suppose $(\Omega,\mathcal{B},\pi),\hat p_t$ and $(\Omega',\mathcal{B}',\pi'),\hat p'_t$ are pseudofinite Markov chains with the same density array.  Then there is an ordering of the eigenvectors of $\hat p_1$, $\nu_1,\ldots$, and the eigenvectors of $\hat p'_1$, $\nu'_1,\ldots$ so that for any finite set $n$, the distribution of $(\nu_1,\ldots,\nu_n)$ is the same as the distribution of $(\nu'_1,\ldots,\nu'_n)$.
  \end{lemma}
  \begin{proof}
For each positive integer $K$, let
\[\hat p^K_1(\omega,\omega')=\left\{\begin{array}{ll}
K&\text{if }\hat p_1(\omega,\omega')>K\\
-K&\text{if }\hat p_1(\omega,\omega')<-K\\
\hat p_1(\omega,\omega')&\text{otherwise}
\end{array}\right..\]
Define $\hat p^{'K}_1$ similarly.  Since the distributions of $\hat p^K_1$ and $\hat p^{'K}_1$ are the same, they have the same moments in the sense of \cite{MR2594615}, and so by the main result of that paper, they are weakly isomorphic: there are measure-preserving, measurable embeddings $\rho,\rho'$ into some common space $(\Omega^K_*,\mathcal{B}^K_*,\pi^K_*),\hat p^K_{1,*}$ so that $\hat p^K_1(\omega,\omega')=\hat p^K_{1,*}(\rho(\omega),\rho(\omega'))$ almost everywhere and $\hat p^{',K}_1(\omega,\omega')=\hat p^K_{1,*}(\rho'(\omega),\rho'(\omega'))$ almost everywhere.

As noted above, each $\hat p^K_{1,*}$ has an eigenvector decomposition, and the pullbacks of the eigenvectors of $\hat p_{1,*}$ under $\rho$ and $\rho'$ must be eigenvectors of $\hat p^K_1$ and $\hat p^{',K}_1$.  In particular, fixing an enumeration of the eigenvectors of $\hat p^K_{1,*}$, the pullback of the first $n$ eigenvectors to $\hat p^K_1$ has the same distribution as the pullback of the first $n$ eigenvectors eigenvector to $\hat p^{',K}_1$.

By choosing $K$ larger and larger, the $i$-th eigenvector of $\hat p^K_1$ converges in $L^2$ norm, so also in distribution, to the $i$-th eigenvector of $\hat p_1$, and similarly for $\hat p'_1$.  This means that for any $n$, the first $n$ eigenvectors of $\hat p_1$ have the same distribution as the first $n$ eigenvectors of $\hat p'_1$.
  \end{proof}

\begin{theorem}\label{thm:uniqueness}
  Suppose $(\Omega,\mathcal{B},\pi),\hat p_t$ and $(\Omega',\mathcal{B}',\pi'),\hat p'_t$ are twin-free saturated pseudofinite Markov chains with the same density array.  Suppose $\mathcal{B}$ and $\mathcal{B}'$ are, respectively, the smallest $\sigma$-algebras which make $\hat p_t$ and $\hat p'_t$ measurable.  Then there is a measure-preserving $\phi:\Omega\rightarrow\Omega'$ which is a bijection up to sets of measure $0$.
\end{theorem}
\begin{proof}
By the previous lemma, we may order the eigenvectors of $(\Omega,\mathcal{B},\pi),\hat p_t$ and $(\Omega',\mathcal{B}',\pi'),\hat p'_t$ so that for each $n$, the first $n$ eigenvectors of $\hat p_1$ have the same distribution as the first $n$ eigenvectors of $\hat p'_1$.  (Note that $\phi$ need not be unique: there could be multiple ways of ordering the eigenvectors so the distributions match up.)

In particular, this means that a type is wide in $\hat p_1$ iff it is wide in $\hat p'_1$.  Almost every point $\omega$ of $\Omega$ is a point with a wide type, and there is exactly one point $\omega'\in\Omega'$ with the same type, so we set $\phi(\omega)=\omega'$.  This is injective (distinct points have distinct types by the twin-freeness of $\Omega$) and surjective up to measure $0$ (almost every point in $\Omega'$ has wide type).  Measurability and the measure-preserving property follows since the inverse image of the set of $I,\epsilon$-almost realizers in $\Omega'$ of some wide type is precisely the the $I,\epsilon$-almost realizers in $\Omega$.  Since $\nu_i(\phi(\omega))=\nu_i(\omega)$ for all $i$ and almost all $\omega$, $\hat p_t(\omega,\omega')=\hat p'_t(\phi(\omega),\phi(\omega'))$ almost everywhere as desired.
\end{proof}

\bibliographystyle{spmpsci}
\bibliography{../../Bibliographies/main}
\end{document}